\documentclass{ieeeaccess}
\usepackage{cite}
\usepackage{amsmath,amssymb,amsfonts,amsthm}
\usepackage{algorithmic}
\usepackage{graphicx}
\usepackage{textcomp}
\usepackage{hyperref}

\DeclareMathOperator*{\esssup}{ess\,sup}
\DeclareMathOperator*{\essinf}{ess\,inf}

\DeclareMathOperator*{\arginf}{arg\,inf}
\newtheorem{theorem}{Theorem}%[section]

\newtheorem{corollary}[theorem]{Corollary}

\newtheorem{definition}[theorem]{Definition}
\newtheorem{claim}[theorem]{Claim}
\newtheorem{remark}[theorem]{Remark}
\newtheorem{example}[theorem]{Example}

\def\BibTeX{{\rm B\kern-.05em{\sc i\kern-.025em b}\kern-.08em
    T\kern-.1667em\lower.7ex\hbox{E}\kern-.125emX}}
\begin{document}
\history{Date of publication xxxx 00, 0000, date of current version xxxx 00, 0000.}
\doi{10.1109/ACCESS.2017.DOI}

\title{Refined Pinsker's and reverse Pinsker's inequalities for probability distributions of different dimensions}
\author{\uppercase{Michele Caprio}\authorrefmark{1}}
\address[1]{PRECISE Center, Department of Computer and Information  Science, 
University of Pennsylvania, 3330 Walnut Street, Philadelphia, PA 19104 USA (e-mail: caprio@seas.upenn.edu, ORCID iD: 0000-0002-7569-097X)}
\tfootnote{This work was supported in part by the National Science Foundation (CCF 1934964) and by the Army Research Office (ARO MURI W911NF2010080).}

\markboth
{Caprio: Pinsker's and reverse Pinsker's for distributions of different dimensions}
{Caprio: Pinsker's and reverse Pinsker's for distributions of different dimensions}

\corresp{Corresponding author: Michele Caprio (e-mail: caprio@seas.upenn.edu).}

\begin{abstract}
We provide optimal lower and upper bounds for the augmented Kullback-Leibler divergence in terms of the augmented total variation distance between two probability measures defined on two Euclidean spaces having different dimensions. We call them refined Pinsker's and reverse Pinsker's inequalities, respectively.
\end{abstract}

\begin{keywords}
Kullback-Leibler divergence, total variation distance, optimal bounds, probability measures of different dimensions.
\end{keywords}

\titlepgskip=-15pt

\maketitle

\section{Introduction}\label{intro}
Bounding the Kullback-Leibler (KL) divergence between probability measures (pm's) defined on the same space in terms of their total variation (TV) distance is a well studied problem, of paramount importance in statistics and machine learning. Famous lower bounds are given by Pinsker's inequality \cite{csiszar} and Vajda's lower bound \cite{vajda}, while a famous upper bound is given by reverse Pinsker's inequality \cite{geiger,sason}. These results are particularly useful in Bayesian nonparametrics \cite{binette2} and in the optimal quantization of pm's \cite{geiger}. 

In this note, we generalize results from \cite{binette,fedotov} to find the optimal (defined below) lower and upper bounds for the KL divergence between pm's defined on two Euclidean spaces having different dimensions in terms of their TV distance. The generalizations of KL divergence and TV distance to pm's of different dimensions are called augmented KL divergence (AKL) and augmented total variation distance (ATV), respectively, and were first introduced in \cite{lim2}. The AKL and the ATV could be used to measure the loss of information after projecting a probability measure $P$ down to a lower-dimensional subspace, e.g. via principal component analysis (PCA). That is, an interesting open research question is to determine whether the larger the AKL or the ATV between $P$ and its projection $\text{Proj}(P)$, the more likely it is to lose information in the projecting process, and if such loss depends on the projection we use. Another interesting information-theoretic application of AKL is the following: it can be used to calculate the divergence between two different dimensional distributions in the field of multi-target labeled probability distributions of a hybrid of continuous state and discrete label variables \cite[Remark 3 and Equation (63)]{li_paper}.

The main result of this paper, Theorem \ref{main3}, states that for any given value $\delta$ of the ATV between two generic distributions defined on Euclidean spaces having different dimensions, we can give optimal lower and upper bounds to their augmented KL divergence.\footnote{As we shall see, the upper bound requires a mild assumption to hold.} An interesting byproduct of Theorem \ref{main3}, explored in Example \ref{ex1}, is that we can also give optimal bounds to ATV in terms of (a fixed value of) the AKL. Notice also that, paraphrasing \cite[Section I]{fedotov}, knowing the relation between AKL and ATV enables to translate results from information theory -- results involving the AKL -- to results in probability theory -- results involving the ATV -- and vice versa.

When $P$ and $Q$ are defined on the same space, ``optimality'' should be understood as follows. For the refined Pinsker's inequality, we mean the best lower bound on the KL divergence between $P$ and $Q$ given that their TV distance is some fixed value $\delta\geq 0$, that is, $\inf_{d_{TV}(P,Q)=\delta}D_{KL}(P\|Q)$. For the refined reverse Pinsker's inequality, we mean the best upper bound on the KL divergence between $P$ and $Q$ over the class $\mathcal{A}(\delta,m,M)$ of pm's whose TV distance is equal to $\delta\geq 0$ and whose relative density $\text{d}P/\text{d}Q$ has finite lower and upper bounds $m$ and $M$, respectively, introduced in Definition \ref{tight_rev_pins}.\footnote{The concept of relative density will be introduced in section \ref{prelim}.} That is, $\sup_{(P,Q)\in\mathcal{A}(\delta,m,M)}D_{KL}(P\|Q)$. As we can see, the meaning of ``optimality'' for the upper bound is slightly less general than that for the lower bound. As pointed out in \cite[Section 1]{sason}, this is due to the fact that for any $\varepsilon>0$, there exists a pair $P,Q$ of pm's such that $d_{TV}(P,Q)\leq \varepsilon$ while $D_{KL}(P\|Q)=\infty$. Consequently, a reverse Pinsker's inequality which provides an upper bound on the KL divergence between $P$ and $Q$ when their TV distance is some fixed $\delta\geq 0$ may not exist in general, whence the necessity of working with $\mathcal{A}(\delta,m,M)$. The generalizations of these ``optimality'' concepts to AKL and ATV are given in section  \ref{main.section}.

The note is divided as follows. Section \ref{prelim} gives the needed background, and section \ref{main.section} presents our main result. Section \ref{concl} is a discussion.

\section{Preliminaries}\label{prelim}
\subsection{Probability measures on the same measurable space}
Pick two pm's $P,Q$ defined on the same measurable space $(\Omega,\mathcal{F})$ and assume $P$ is absolutely continuous with respect to $Q$, written $P\ll Q$. This means that $Q(A)=0$ implies $P(A)=0$, $A\in\mathcal{F}$. Denote by $\text{d}P/\text{d}Q$ the relative density of $P$ with respect to $Q$, that is, $\text{d}P/\text{d}Q \equiv \mathfrak{f}$ is an $\mathcal{F}$-measurable functional on $\Omega$ such that for all $A\in\mathcal{F}$,
$$P(A)=\int_A \mathfrak{f} \text{ d}Q.$$
Then, the KL divergence and the TV distance between $P$ and $Q$ are defined as
\begin{align}\label{dtv}
\begin{split}
    &D_{KL}(P\|Q):=\int_\Omega \log\left( \frac{\text{d}P}{\text{d}Q}  \right)\text{d}P\\ 
    \text{and} \quad &d_{TV}(P,Q):=\sup_{A\in\mathcal{F}} \left| P(A)-Q(A) \right|,
    \end{split}
\end{align}
respectively.\footnote{We do not need the absolute continuity assumption to hold for the TV metric.} Consider the following function, that -- given some $\delta\geq 0$ -- selects the smallest possible value of the KL divergence between pm's whose TV distance is equal to $\delta$ 
\begin{equation}\label{tight_pinsker}
   \delta \mapsto L(\delta):=\inf_{d_{TV}(P,Q)=\delta} D_{KL}(P\| Q). 
\end{equation}
It is called the \textit{Vajda's lower bound} \cite{vajda}. The following comes from \cite[Theorem 1]{fedotov}.
\begin{theorem}\label{fedotov}\textbf{(Fedotov, Harremoës, and Topsøe)}
Pick two probability measures $P,Q$ defined on a generic measurable space $(\Omega,\mathcal{F})$ and assume $d_{TV}(P,Q)=\delta\geq 0$. Then, curve $\gamma:\delta \mapsto (\delta,L(\delta))$ is a differentiable curve in the $(d_{TV},D_{KL})$-plane, symmetric around the $D_{KL}$-axes. 
%With $t=\frac{\text{d}L}{\text{d}\delta}$, the relationship $t \leftrightarrow \delta$ is a diffeomorphism between $\mathbb{R}$ and $(-2,2)$. 
In addition, using $t=\frac{\text{d}L}{\text{d}\delta}\in\mathbb{R}_+$ as a parameter, $\gamma$ is parametrized by
\begin{align}\label{param_gamma}
\begin{split}
    \delta(t)&=t\left(1-\left(\coth(t)-\frac{1}{t} \right)^2 \right)\\
    L(\delta(t)&)=\log\left( \frac{t}{\sinh (t)} \right) + t\coth (t) -\frac{t^2}{\sinh^2(t)}.
\end{split}
\end{align}
\end{theorem}

In \cite[Corollary 1]{reid}, the authors give an explicit value for $L(\delta)$, in contrast with \eqref{param_gamma} where the value is implicit.

\begin{corollary}\label{reid_cor}\textbf{(Reid and Williamson)}
Pick two probability measures $P,Q$ defined on a generic measurable space $(\Omega,\mathcal{F})$ and assume $d_{TV}(P,Q)=\delta\geq 0$. Then,
\begin{align}\label{reid_eq}
    L(\delta) = \min_{\gamma\in[\delta-2,2-\delta]} \bigg[ &\left( \frac{\delta+2-\gamma}{4} \right) \log\left( \frac{\gamma-2-\delta}{\gamma-2+\delta} \right) \nonumber \\
    + &\left( \frac{\gamma+2-\delta}{4} \right) \log\left( \frac{\gamma+2-\delta}{\gamma+2+\delta} \right) \bigg].
\end{align}
\end{corollary}

We now define $\mathcal{A}(\delta,m,M)$, a set of pairs of probabilities that will be useful in the rest of the work. Before doing so, we need to introduce the concepts of essential infimum $\essinf \mathfrak{f}$ and essential supremum $\esssup \mathfrak{f}$ of $\text{d}P/\text{d}Q \equiv \mathfrak{f}$ with respect to $Q$. We have that
\begin{align*}
    \essinf \mathfrak{f} &:= \sup \left\lbrace{b\in\mathbb{R} : Q\left( \left\lbrace{\omega\in\Omega : \mathfrak{f}(\omega)<b}\right\rbrace \right)=0}\right\rbrace,\\
    \esssup \mathfrak{f} &:= \inf \left\lbrace{a\in\mathbb{R} : Q\left( \left\lbrace{\omega\in\Omega : \mathfrak{f}(\omega)>a}\right\rbrace \right)=0}\right\rbrace.
\end{align*}

\begin{definition}\label{tight_rev_pins}
Fix $\delta\geq 0$, $m>0$, and $M<\infty$. We call $\mathcal{A}(\delta,m,M)$ the set of all pm's pairs $(P,Q)$ defined on a common measurable space $(\Omega,\mathcal{F})$ satisfying
\begin{enumerate}
    \item $P\ll Q$,
    \item $\essinf \frac{\text{d}P}{\text{d}Q}=m$,
    \item $\esssup \frac{\text{d}P}{\text{d}Q}=M$,
    \item $d_{TV}(P,Q)=\delta$.
    %\in (-2,2)
    \end{enumerate}
%Here $\essinf$ and $\esssup$ represent the essential infimum and supremum taken with respect to $Q$.
\end{definition}

The optimal upper bound for the KL divergence between (a pair of) pm's belonging to $\mathcal{A}(\delta,m,M)$ is defined as
\begin{equation}\label{opt_u_b}
    U(\mathcal{A}(\delta,m,M)):=\sup_{(P,Q)\in\mathcal{A}(\delta,m,M)}D_{KL}(P\|Q).
\end{equation}

We have the following important result.

\begin{theorem}\label{first_main}
Pick any $\delta\geq 0$, $m>0$, $M<\infty$, and assume $\mathcal{A}(\delta,m,M)\neq \emptyset$. Then, for all $(P,Q)\in\mathcal{A}(\delta,m,M)$, the following are optimal bounds
\begin{equation}\label{bds}
    L(\delta) \leq D_{KL}(P\| Q) \leq U(\mathcal{A}(\delta,m,M)).
\end{equation}
\end{theorem}
\begin{proof}
The optimal upper bound for $D_{KL}(P\| Q)$ comes from equation \eqref{opt_u_b}. Its value, given in \cite[Equation (9)]{binette}, is
\begin{equation}\label{opt_u_b_value}
    U(\mathcal{A}(\delta,m,M)) = \delta \left( \frac{\log(M^{-1})}{1-M^{-1}} + \frac{\log(m^{-1})}{m^{-1}-1}  \right).
\end{equation}
The optimal lower bound comes from equation \eqref{param_gamma}. An implicit parametric solution of the form
of the graph of Vajda's lower bound as $(V (t), L(t))_{t\in\mathbb{R}_+}$ is given in Theorem \ref{fedotov}, while an explicit value for $L(\delta)$ is given in Corollary \ref{reid_cor}.
\end{proof}
Notice that in the case where $m=1$ or $M=1$, any $(P,Q)\in\mathcal{A}(\delta,m,M)$ must be such that $\delta=d_{TV}(P,Q)=0$. The right hand side of \eqref{opt_u_b_value} is then understood as being equal to $0$. In addition, the assumption that the pair $(P,Q)$ belongs to $\mathcal{A}(\delta,m,M)$ is only needed to obtain the upper bound in \eqref{bds}, as pointed out in section \ref{intro}.

In \cite[Theorem 7]{fedotov}, the authors find a lower bound for $L(\delta)$ that makes computing a lower bound for the KL divergence in terms of the TV metric easier.
\begin{theorem}\label{thm2}\textbf{(Fedotov, Harremoës, and Topsøe)}
Pick two probability measures $P,Q$ defined on a generic measurable space $(\Omega,\mathcal{F})$ and assume $d_{TV}(P,Q)=\delta\geq 0$. Then, the following is true
$$L(\delta)\geq \frac{1}{2}\delta^2+\frac{1}{36}\delta^4+\frac{1}{270}\delta^6+\frac{221}{340200}\delta^8.$$
\end{theorem}

\subsection{Probability measures on two Euclidean spaces with different dimensions}
In this paper, we adopt the framework of \cite{lim2} to prove a version of Theorems \ref{first_main} and \ref{thm2} for pm's pairs $(P,Q)$ defined on two Euclidean spaces having different dimensions. 
%belonging to $\Delta(\mathbb{R}^d,\mathcal{B}(\mathbb{R}^d)) \times \Delta(\mathbb{R}^n,\mathcal{B}(\mathbb{R}^n))$, where $\Delta(\mathbb{R}^\bullet,\mathcal{B}(\mathbb{R}^\bullet))$ denotes the space of probability measures on measurable space $(\mathbb{R}^\bullet,\mathcal{B}(\mathbb{R}^\bullet))$, and $\mathcal{B}(\mathbb{R}^\bullet)$ is the Borel sigma-algebra on $\mathbb{R}^\bullet$. and let $M^p(\Omega)\subset M(\Omega)$ denote those with finite $p$-th moments, $p\in\mathbb{N}$
Let $M(\Omega)$ denote the set of all Borel pm's on $\Omega \subset \mathbb{R}^n$. For convenience, we restrict our attention to pm's with densities so that we do not have to keep track of which measure is absolutely continuous to which other measure \cite[Section III]{lim2}; this is without loss of generality. Let $\lambda^n$ be the Lebesgue measure restricted to $\Omega\subset \mathbb{R}^n$. With respect to $\lambda^n$, we define 
\begin{align*}
    M_{dens}(\Omega):=\{\mu\in M(\Omega):\mu \text{ has density}\}.
\end{align*}
Notice that $\mu\in M_{dens}(\Omega)$ if and only if it is absolutely continuous with respect to $\lambda^n$. The Lebesgue measure is chosen because it is the most common measure; it can be substituted by any measure satisfying the condition that for any nonzero area, the measure of said area is positive. This requirement is needed to make $D^-_{KL}$, $D^+_{KL}$, $d^-_{TV}$, and $d^+_{TV}$ in Theorem \ref{cai_lim} well defined.

%M_{pd}(\Omega)&:=\{\mu\in M_{dens}(\Omega):\mu \text{ has strictly positive density}\}.

We now introduce the machinery that we use to project a pm to a lower dimensional space and to embed a pm to a higher dimensional space.  For any $d,n \in\mathbb{N}$, $d\leq n$, let
$$O(d,n):=\{V\in\mathbb{R}^{d\times n} : VV^\top=I_d\},$$
that is, the \textit{Stiefel manifold} of $d\times n$ matrices with orthonormal rows. 
%We write $O(n):=O(n,n)$ for the orthogonal group. 
For any $V\in O(d,n)$ and $b\in \mathbb{R}^d$, let
$$\varphi_{V,b}:\mathbb{R}^n \rightarrow \mathbb{R}^d, \quad x \mapsto \varphi_{V,b}(x):=Vx+b,$$
and for any $\mu\in M(\mathbb{R}^n)$, let $\varphi_{V,b}(\mu)\equiv {\varphi_{V,b}}_\sharp\mu$ be the pushforward of measure $\mu$ through function $\varphi_{V,b}$. That is, for every element $A$ of the sigma-algebra endowed to $\mathbb{R}^d$, $\varphi_{V,b}(\mu)(A)\equiv {\varphi_{V,b}}_\sharp\mu(A):=\mu(\varphi_{V,b}^{-1}(A))$.
\begin{definition}\label{def1}
Let $d,n\in\mathbb{N}$, $d\leq n$. For any $P\in M(\mathbb{R}^d)$ and $Q\in M(\mathbb{R}^n)$, the set of  embeddings of $P$ into $\mathbb{R}^n$ is
\begin{align*}
    \Phi^+(P,n):=\{&\alpha\in M(\mathbb{R}^n): \varphi_{V,b}(\alpha)=P,\\ 
    &\text{for some } V\in O(d,n), b\in \mathbb{R}^d\}
\end{align*}
and the set of projections of $Q$ onto $\mathbb{R}^d$ is 
\begin{align*}
    \Phi^-(Q,d):=\{&\beta\in M(\mathbb{R}^d): \varphi_{V,b}(Q)=\beta,\\
    &\text{for some } V\in O(d,n), b\in \mathbb{R}^d\}.
\end{align*}
\end{definition}

\begin{remark}
Definition \ref{def1} is stating the following. The set of embeddings of a probability measure $P$ (defined on $\mathbb{R}^d$) onto $\mathbb{R}^n$, $n\geq d$, is given by those probabilities on $\mathbb{R}^n$ whose pushforward through function $\varphi_{V,b}$ recovers $P$, for some $V\in O(d,n)$ and $b\in \mathbb{R}^d$. The set of projections of a probability measure $Q$ (defined on $\mathbb{R}^n$) onto $\mathbb{R}^d$, $n\geq d$, is given by those probabilities on $\mathbb{R}^d$ that can be written as the pushforward of $Q$ through function $\varphi_{V,b}$, for some $V\in O(d,n)$ and $b\in \mathbb{R}^d$.
\end{remark}

An important subset of $\Phi^+(P,n)$ is
\begin{align*}
    \Phi^+_{dens}(P,n):=\{&\alpha\in M_{dens}(\mathbb{R}^n): \varphi_{V,b}(\alpha)=P,\\
&\text{for some } V\in O(d,n), b\in \mathbb{R}^d\}.
\end{align*}

The following relevant result comes from \cite[Theorem III.4]{lim2}.
\begin{theorem}\label{cai_lim}\textbf{(Cai and Lim)}
Let $d,n\in\mathbb{N}$, $d\leq n$. For any $P\in M(\mathbb{R}^d)$ and $Q\in M(\mathbb{R}^n)$, let
\begin{align*}
    D_{KL}^-(P\| Q)&:=\inf_{\beta\in\Phi^-(Q,d)}D_{KL}(P\| \beta),\\
    D_{KL}^+(P\| Q)&:=\inf_{\alpha\in\Phi^+_{dens}(P,n)}D_{KL}(\alpha\| Q),\\
    d_{TV}^-(P, Q)&:=\inf_{\beta\in\Phi^-(Q,d)}d_{TV}(P, \beta),\\
    d_{TV}^+(P, Q)&:=\inf_{\alpha\in\Phi^+_{dens}(P,n)}d_{TV}(\alpha, Q).
\end{align*}
Then,
$$D_{KL}^-(P\| Q)=D_{KL}^+(P\| Q)\equiv \hat{D}_{KL}(P\| Q)$$
and
$$d_{TV}^-(P, Q)=d_{TV}^+(P, Q)\equiv \hat{d}_{TV}(P, Q).$$
\end{theorem}

We call $\hat{D}_{KL}(P\| Q)$ the \textit{augmented KL divergence} (AKL), while $\hat{d}_{TV}(P, Q)$ the \textit{augmented TV distance} (ATV). Notice that \cite[Lemma III.2]{lim2} guarantees the existence of quantities $D_{KL}^-(P\| Q)$, $D_{KL}^+(P\| Q)$, $d_{TV}^-(P, Q)$, and $d_{TV}^+(P, Q)$.

\section{Main result}\label{main.section}
Consider function
\begin{equation}\label{L_hat}
    \delta \mapsto \hat{L}(\delta):=\inf_{\hat{d}_{TV}(P,Q)=\delta} \hat{D}_{KL}(P\| Q).
\end{equation}
Being the augmented counterpart of \eqref{tight_pinsker}, we call it the \textit{augmented Vajda's lower bound}. Denote by $\boldsymbol{\alpha} \in \Phi^+_{dens}(P,n)$ and $\boldsymbol{\beta} \in \Phi^-(Q,d)$ the pm's such that $\hat{D}_{KL}(P \| Q)={D}_{KL}(\boldsymbol{\alpha}\| Q)={D}_{KL}(P \| \boldsymbol{\beta})$, that is, 
\begin{align}\label{arginfs}
     &\boldsymbol{\alpha}=\arginf_{\beta\in\Phi^+_{dens}(P,n)}D_{KL}(\alpha\|Q) \nonumber\\
     \text{and } \quad &\boldsymbol{\beta}=\arginf_{\beta\in\Phi^-(Q,d)}D_{KL}(P\|\beta).
\end{align}
Let then
\begin{align}
    \essinf \frac{\text{d}\boldsymbol{\alpha}}{\text{d}Q}=m_1,\quad
    \esssup \frac{\text{d}\boldsymbol{\alpha}}{\text{d}Q}=M_1, \label{ess1}\\
    \essinf \frac{\text{d}P}{\text{d}\boldsymbol{\beta}}=m_2,\quad
    \esssup \frac{\text{d}P}{\text{d}\boldsymbol{\beta}}=M_2. \label{ess2}
\end{align}
Notice that \eqref{ess1} are taken with respect to $Q$, while \eqref{ess2} are taken with respect to $\boldsymbol{\beta}$. They correspond to (2) and (3) in Definition \ref{tight_rev_pins}. We need to bound the relative densities $\text{d}\boldsymbol{\alpha}/\text{d}Q$ and $\text{d}P/\text{d}\boldsymbol{\beta}$ otherwise we may have that $\hat{d}_{TV}(P,Q)= \delta$, but $\hat{D}_{KL}(P\|Q)=\infty$, similarly to what we pointed out in section \ref{intro}. We now define a set of pairs of probabilities that is the augmented counterpart of Definition \ref{tight_rev_pins}.
\begin{definition}\label{tight_rev_pins2}
Pick $d,n\in\mathbb{N}$ such that $d\leq n$. Fix $\delta\geq 0$, $m_1,m_2>0$ and $M_1,M_2<\infty$. $\mathcal{A}(\delta,m_1,m_2,M_1,M_2)$ is the set of all pm's pairs $(P,Q)$ in $M(\mathbb{R}^d)\times M(\mathbb{R}^n)$ such that
%\footnote{We denote by $\mathcal{B}(\mathbb{R}^J)$ the Borel sigma-algebra on $\mathbb{R}^J$, $J\in\mathbb{N}$.}
\begin{enumerate}
    \item[(i)] $\boldsymbol{\alpha}\ll Q$ and $P\ll \boldsymbol{\beta}$,
    \item[(ii)] \eqref{ess1} and \eqref{ess2} are satisfied,
    \item[(iii)] $\hat{d}_{TV}(P, Q)=\delta$.
\end{enumerate}
\end{definition}

The optimal upper bound for the AKL between (a pair of) pm's belonging to the set  $\mathcal{A}(\delta,m_1,m_2,M_1,M_2)$ is defined as
\begin{align}\label{opt_u_b_aug}
    \hat{U}(\mathcal{A}(\delta,m_1&,m_2,M_1,M_2)) \nonumber\\
    &:=\sup_{(P,Q)\in\mathcal{A}(\delta,m_1,m_2,M_1,M_2)} \hat{D}_{KL}(P\|Q).
\end{align}

The following is our main result.

\begin{theorem}\label{main3}
Pick $d,n\in\mathbb{N}$ such that $d\leq n$. Fix $\delta\geq 0$, $m_1,m_2> 0$, and $M_1,M_2<\infty$. Assume $\mathcal{A}(\delta,m_1,m_2,M_1,M_2) \neq \emptyset$. Pick any $(P,Q)$ in $\mathcal{A}(\delta,m_1,m_2,M_1,M_2)$ and let 
\begin{align*}
    \text{pol}_{\hat{d}_{TV}}&:=\frac{1}{2}\hat{d}_{TV}(P,Q)^2+\frac{1}{36}\hat{d}_{TV}(P,Q)^4\\
    &+\frac{1}{270}\hat{d}_{TV}(P,Q)^6+\frac{221}{340200}\hat{d}_{TV}(P,Q)^8.
\end{align*}
Then,
\begin{align}\label{bds2}
    \text{pol}_{\hat{d}_{TV}} &\leq \hat{L}(\delta) \nonumber\\ 
    &\leq \hat{D}_{KL}(P\| Q) \leq \hat{U}(\mathcal{A}(\delta,m_1,m_2,M_1,M_2)).
\end{align}
%\begin{align}\label{main.equation}
%\begin{split}
    %\frac{1}{2}&\hat{d}_{TV}(P,Q)^2+\frac{1}{36}\hat{d}_{TV}(P,Q)^4+\frac{1}{270}\hat{d}_{TV}(P,Q)^6+\frac{221}{340200}\hat{d}_{TV}(P,Q)^8 \leq \hat{L}(\delta) \\&\leq \hat{D}_{KL}(P\| Q)\\ &\leq \min\left\lbrace{ \hat{d}_{TV}(P,Q)\left( \frac{\log(M_{1}^{-1})}{M_{1}^{-1}-1} + \frac{\log(m_{1}^{-1})}{m_{1}^{-1}-1}  \right) , \hat{d}_{TV}(P,Q)\left( \frac{\log(M_{2}^{-1})}{M_{2}^{-1}-1} + \frac{\log(m_{2}^{-1})}{m_{2}^{-1}-1}  \right)  }\right\rbrace.
%\end{split}
%\end{align}
\end{theorem}

Before proving our result, let us remark that assuming $(P,Q)\in\mathcal{A}(\delta,m_1,m_2,M_1,M_2)$ is only needed to upper bound $\hat{D}_{KL}(P\| Q)$. The reason is that otherwise such upper bound may not exist, as pointed out earlier in this section. In addition, the second and the third inequalities in \eqref{bds2} are optimal. Finally, notice that there is an elegant relationship between Theorem \ref{thm2} and the first inequality in \eqref{bds2}. We can lower bound Vajda's bound $L(\delta)$ and the augmented Vajda's bound $\hat{L}(\delta)$ by the same polynomial, the first one in $\delta=d_{TV}(P,Q)$ and the second one in $\delta=\hat{d}_{TV}(P,Q)$.

%Just like in Theorem \ref{thm2} a lower bound for Vajda's lower bound is given by a polynomial in $d_{TV}(P,Q)=\delta$, the first inequality in Theorem \ref{main3} gives a lower bound for the augmented Vajda's lower bound 

\begin{proof}
The proof has four steps.
\begin{enumerate}
    \item[(I)] We first show that $\text{pol}_{\hat{d}_{TV}} \leq \hat{D}_{KL}(P\| Q)$. We have that
    \begin{align*}
  \hat{D}_{KL}&(P\| Q)=\inf_{\beta\in\Phi^-(Q,d)} D_{KL}(P\| \beta)\\
  &\geq \inf_{\beta\in\Phi^-(Q,d)} \bigg[ \frac{1}{2} {d}_{TV}(P,\beta)^2+\frac{1}{36}{d}_{TV}(P,\beta)^4\\
  &+\frac{1}{270}{d}_{TV}(P,\beta)^6+\frac{221}{340200}{d}_{TV}(P,\beta)^8 \bigg]\\
  &\geq \frac{1}{2} \hat{d}_{TV}(P,\beta)^2+\frac{1}{36}\hat{d}_{TV}(P,\beta)^4\\ &+\frac{1}{270}\hat{d}_{TV}(P,\beta)^6+\frac{221}{340200}\hat{d}_{TV}(P,\beta)^8.
\end{align*}
Here, the equality comes from Theorem \ref{cai_lim}, the first inequality is a consequence of Theorems \ref{first_main} and \ref{thm2}, and the second inequality comes from Theorem \ref{cai_lim} and the fact that the infimum of a sum is not smaller than the sum of the infima. Notice that if we substitute $\inf_{\beta\in\Phi^-(Q,d)}$ with $\inf_{\alpha\in\Phi^+_{dens}(P,n)}$ the proof still holds thanks to Theorem \ref{cai_lim}.

\item[(II)] The fact that $\hat{D}_{KL}(P\| Q) \geq \hat{L}(\delta)$ comes from equation \eqref{L_hat} and the assumption that $\hat{d}_{TV}(P,Q)=\delta$. We also have the following result.
\begin{claim}\label{claim_str_lb}
A version of parametrization \eqref{param_gamma} holds for $\hat{L}(\delta)$. Let then 
$$\boldsymbol{\beta}^\prime=\arginf_{\beta\in\Phi^-(Q,d)}d_{TV}(P,\beta).$$ 
If $\boldsymbol{\beta}=\boldsymbol{\beta}^\prime$, a version of equation \eqref{reid_eq} holds for $\hat{L}(\delta)$. 
\end{claim}

\begin{proof}
To prove the first part of the claim, we begin by showing that $\hat{D}_{KL}$ is convex, jointly in $P$ and $Q$. To see this, notice that, given two generic probability measures $P,Q$ on the same measurable space $(\Omega,\mathcal{F})$, \cite[Section II]{fedotov} points out that ${D}_{KL}$ is strictly convex, jointly in $P$ and $Q$. In our case, we have that $\hat{D}_{KL}(P\|Q)=\inf_{\beta\in\Phi^-(Q,d)} {D}_{KL}(P\|\beta)$; because the infimum operator preserves convexity, we can conclude that $\hat{D}_{KL}$ is convex, jointly in $P$ and $Q$. In addition, we have that $\hat{L}(\delta):=\inf_{\hat{d}_{TV}(P,Q)=\delta} \hat{D}_{KL}(P\| Q)$. Given the convexity of $\hat{D}_{KL}$, and since the infimum operator preserves convexity, we can conclude that $\hat{L}$ is convex as well. These convexity results entail that for any $\delta\geq 0$ for which $\hat{d}_{TV}(P,Q)=\delta$, there exists a unique pair $(P_\delta,Q_\delta)\in M(\mathbb{R}^d)\times M(\mathbb{R}^n)$ of probability measures such that $\hat{D}_{KL}(P\|Q)$ is minimal among all distributions with augmented total variation equal to $\delta$. 

The augmented Vajda's lower bound, then, is given by the function $\delta\mapsto \hat{L}(\delta)=\hat{D}_{KL}(P_\delta\|Q_\delta)$. Let now $\hat{\gamma}$ denote the map $\delta \mapsto (\delta,\hat{L}(\delta))$. Parameter $\delta$ cannot be used to give an explicit parametrization of $\hat{\gamma}$. Since both $\hat{D}_{KL}$ and $\hat{L}$ are convex functions, the convex conjugate \cite{rock} of both these functions can be explicitly calculated. To prove the statement, we follow the proof of \cite[Theorem 1]{fedotov}. There, the authors use parameter $t=\mathrm{d}\hat{L}/\mathrm{d}\delta$ from the convex conjugate of $\hat{L}$ to parametrize $\hat{L}$. 

Before going on, we give two remarks. The first one is that in \cite{fedotov} the authors work with the so-called \textit{signed total variation metric} $d_{TV}^s(P,Q):=2\sup_{A\in\mathcal{F}}(P(A)-Q(A)) \in [-2,2]$ between probability measures defined on the same measurable space. This is merely a convenience choice (it is easier to obtain parametrization \eqref{param_gamma}), since
$$d_{TV}(P,Q)=\begin{cases}
\frac{1}{2} d_{TV}^s(P,Q) & \text{ if } d_{TV}^s(P,Q)\geq 0\\
-\frac{1}{2} d_{TV}^s(P,Q) & \text{ if } d_{TV}^s(P,Q)< 0
\end{cases}.$$
%This is slightly more general than the total variation metric introduced in \eqref{dtv}. 
Given that $d_{TV}^s$ is an $f$-divergence \cite[Section I]{lim2}, Theorem \ref{cai_lim} holds also if we use   $d_{TV}^s$ in place of $d_{TV}$. In particular,
\begin{align*}
    \hat{d}_{TV}^s(P,Q)&=\inf_{\beta\in\Phi^-(Q,d)}d_{TV}^s(P,\beta)\\
    &=\inf_{\alpha\in\Phi^+_{dens}(P,n)}d_{TV}^s(\alpha,Q).
\end{align*}
Notice that, because
$$\hat{d}_{TV}(P,Q)=\begin{cases}
\frac{1}{2} \hat{d}_{TV}^s(P,Q) & \text{ if } \hat{d}_{TV}^s(P,Q)\geq 0\\
-\frac{1}{2} \hat{d}_{TV}^s(P,Q) & \text{ if } \hat{d}_{TV}^s(P,Q)< 0
\end{cases},$$
in the proof that follows we abuse notation and denote by $\delta$ both the value of $\hat{d}_{TV}(P,Q)$ and that of $\hat{d}_{TV}^s(P,Q)$. The second remark is that in \cite{fedotov} the authors consider a two-elements state space on which $P$ and $Q$ are defined. As they highlight in \cite[Section II]{fedotov}, this simplification is without loss of generality since their results hold even in a continuous or a non-commutative setting. In our more general case, we keep this simplification: we assume that $P$ is defined on the two-elements state space $\Omega=\{\omega_1,\omega_2\}$, so $P=(p_1,p_2=1-p_1)$, $Q$ is defined on a higher-dimensional state space, and set $\Phi^-(Q,d)$ of projections of $Q$ onto $\Omega$ is a subset of $M(\Omega)$. This entails that $\boldsymbol{\beta}=(\beta_1,\beta_2=1-\beta_1)$. 

Let ${D}_{KL}(p_1,\beta_1):=p_1 \log(p_1/\beta_1)$. The convex conjugate of $\hat{D}_{KL}$ is
$$\hat{D}^\star(x,y)=\sup_{p_1,\beta_1} \left(\begin{pmatrix} x\\ y \end{pmatrix}\begin{pmatrix} p_1\\ \beta_1 \end{pmatrix} - {D}_{KL}(p_1,\beta_1) \right).$$
We have
\begin{align*}
    \frac{\partial}{\partial p_1} &\left(\begin{pmatrix} x\\ y \end{pmatrix}\begin{pmatrix} p_1\\ \beta_1 \end{pmatrix} - {D}_{KL}(p_1,\beta_1) \right) \\
    &= x- \log \left(\frac{p_1}{\beta_1}\right) + \log \left(\frac{p_2}{\beta_2}\right)\\
    \frac{\partial}{\partial \beta_1} &\left(\begin{pmatrix} x\\ y \end{pmatrix}\begin{pmatrix} p_1\\ \beta_1 \end{pmatrix} - {D}_{KL}(p_1,\beta_1) \right)\\
    &= y +\frac{p_1}{\beta_1}-\frac{p_2}{\beta_2}
\end{align*}
To find the point where these partial derivatives are $0$, we solve the simultaneous equations
\begin{align*}
    x&= \log \left(\frac{p_1}{\beta_1}\right) - \log \left(\frac{p_2}{\beta_2}\right)\\
    y&=\frac{p_2}{\beta_2}-\frac{p_1}{\beta_1}
\end{align*}
whose solutions are
\begin{align}\label{beta1}
    p_1&=e^x \frac{y+e^x-1}{(e^x-1)^2} \nonumber\\
    \beta_1&=\frac{1}{1-e^x}-\frac{1}{y} ,
\end{align}
$x,y\neq 0$. For\footnote{The use of the augmented signed total variation is clear here; had we used the augmented total variation (as defined in Theorem \ref{cai_lim}) instead, we would have equated $\begin{pmatrix} x\\ y \end{pmatrix}$ to $\begin{pmatrix} 0\\ t \end{pmatrix}$, since $\hat{d}_{TV}(P,Q)\in[0,1]$, for all $(P,Q)\in M(\mathbb{R}^d)\times M(\mathbb{R}^n)$.} 
\begin{align}\label{xy}
    \begin{pmatrix} x\\ y \end{pmatrix}=\begin{pmatrix} -2t\\ 2t \end{pmatrix}
\end{align}
we obtain
$$\begin{pmatrix} x\\ y \end{pmatrix}\begin{pmatrix} p_1\\ \beta_1 \end{pmatrix}=t \delta.$$
Hence,
\begin{align*}
    \hat{D}^\star(-2t,2t)&=\sup_{x,y}\left( t\delta - D_{KL}(p_1,\beta_1)\right)\\
    &=\sup_\delta \left( t\delta - \hat{L}(\delta)\right)
\end{align*}
is the convex conjugate of $\hat{L}$, and $t$ must be the derivative of $\hat{L}$. We see that \eqref{beta1} and \eqref{xy} solve our optimization problem. 
%Convex conjugation transforms differentiable functions into differentiable functions. 
The parametrization of $\hat{\gamma}$
\begin{align}
    \delta(t)&=t\left(1-\left(\coth(t)-\frac{1}{t} \right)^2 \right) \nonumber\\
    \hat{L}(\delta(t)&)=\log\left( \frac{t}{\sinh (t)} \right) + t\coth (t) -\frac{t^2}{\sinh^2(t)} \label{eq_L_hat}
\end{align}
is then obtained by direct evaluation of the quantities involved. A visual representation of $\hat{L}(\delta(t))$ is given in Figure \ref{hat_L_curve}.

\begin{figure}[h!]
\centering
\includegraphics[width=.4\textwidth]{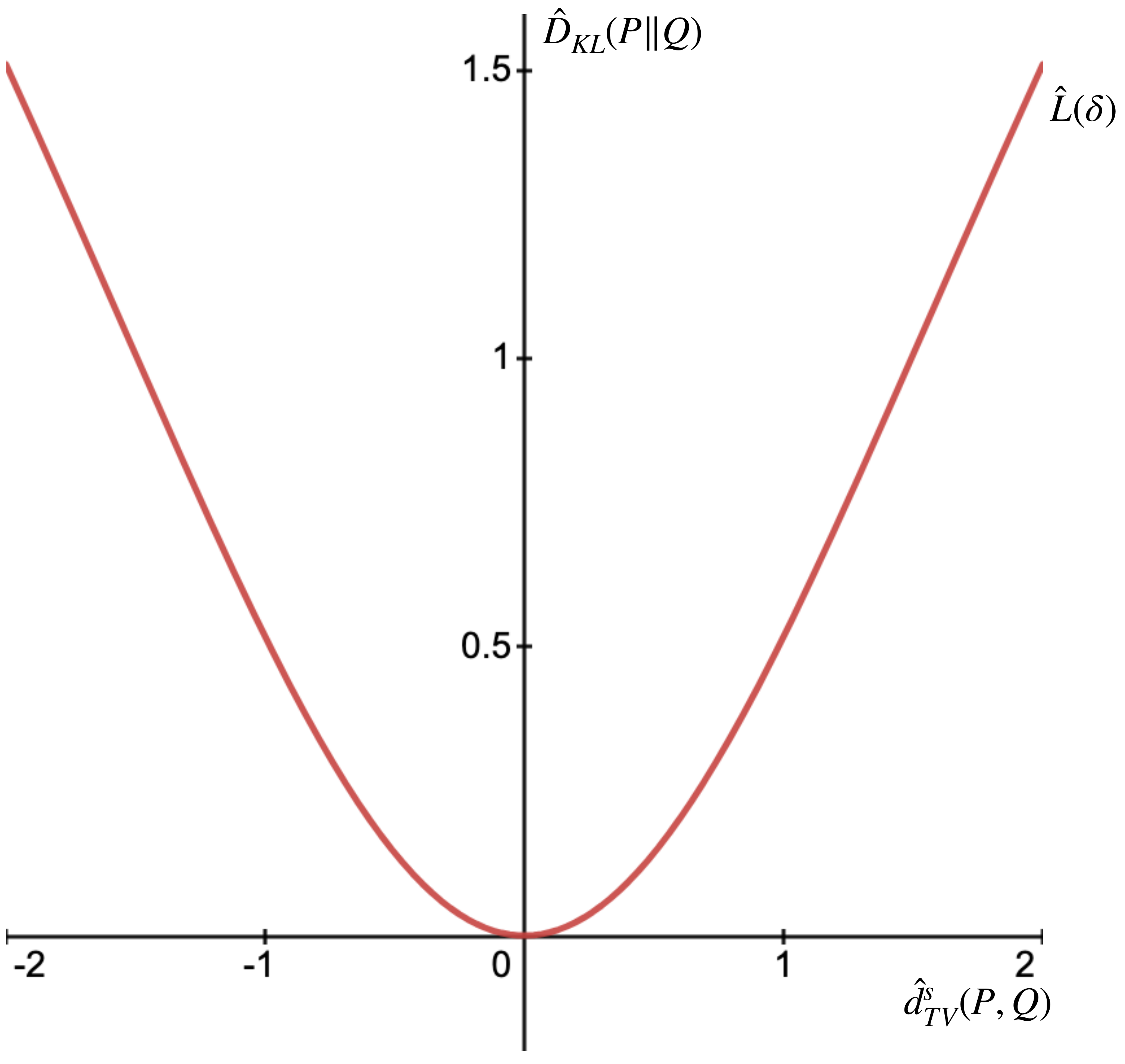}
\caption{A visual representation of $\hat{L}(\delta)$ in \eqref{eq_L_hat}. As we can see, it is symmetric around the $\hat{D}_{KL}(P\|Q)$ axis, which implies that using $\hat{d}^s_{TV}$ in place of $\hat{d}_{TV}$ does not yield any loss of generality.}
\centering
\label{hat_L_curve}
\end{figure}

Suppose now that $\boldsymbol{\beta}=\boldsymbol{\beta}^\prime\equiv\boldsymbol{\beta}^\star$. Then, this implies that we can write $\hat{L}(\delta)$ as $\inf_{d_{TV}(P,\boldsymbol{\beta}^\star)=\delta}D_{KL}(P\|\boldsymbol{\beta}^\star)$. Corollary \ref{reid_cor} entails that for 
$$\hat{d}_{TV}(P,Q)=d_{TV}(P,\boldsymbol{\beta}^\star)=\delta\geq 0,$$ 
we have that
\begin{align}
    \hat{L}&(\delta):=\inf_{\hat{d}_{TV}(P,Q)=\delta} \hat{D}_{KL}(P\| Q) \nonumber\\
    &=\inf_{{d}_{TV}(P,\boldsymbol{\beta}^\star)=\delta} {D}_{KL}(P\| \boldsymbol{\beta}^\star) \nonumber\\
    &=\min_{\gamma\in[\delta-2,2-\delta]} \bigg[ \left( \frac{\delta+2-\gamma}{4} \right) \log\left( \frac{\gamma-2-\delta}{\gamma-2+\delta} \right) \nonumber\\
    &+ \left( \frac{\gamma+2-\delta}{4} \right) \log\left( \frac{\gamma+2-\delta}{\gamma+2+\delta} \right) \bigg]. \label{reid2}
\end{align}
Notice that in this case the parametrization in \eqref{param_gamma} holds too, but it is better to express $\hat{L}(\delta)$ explicitly as in \eqref{reid2}.
\end{proof}

%In addition, since the infimum operator preserves the convexity of a convex function, using $t=\frac{\text{d}\hat{L}}{\text{d}\delta}$ as a parameter and following the steps in the proof of \cite[Theorem 1]{fedotov}, we obtain that curve $\hat{\gamma}:\delta \mapsto (\delta,\hat{L}(\delta))$ is parametrized as in \eqref{param_gamma}.

\item[(III)] The fact that $\hat{D}_{KL}(P\| Q) \leq \hat{U}(\mathcal{A}(\delta,m_1,m_2,M_1,M_2))$ comes from equation \eqref{opt_u_b_aug}. We also have the following result.
\begin{claim}\label{claim_ub_aug}
A version of equation \eqref{opt_u_b_value} holds for $\hat{U}(\mathcal{A}(\delta,m_1,m_2,M_1,M_2))$.
\end{claim}
\begin{proof}
We have that 
\begin{align*}
  &\hat{D}_{KL}(P\| Q)=\inf_{\beta\in\Phi^-(Q,d)} D_{KL}(P\| \beta)\\ 
  &\leq \inf_{\beta\in\Phi^-(Q,d)} {d}_{TV}(P,\beta)\left( \frac{\log(M_{2}^{-1})}{1-M_{2}^{-1}} + \frac{\log(m_{2}^{-1})}{m_{2}^{-1}-1}  \right)\\
  &=\hat{d}_{TV}(P,Q)\left( \frac{\log(M_{2}^{-1})}{1-M_{2}^{-1}} + \frac{\log(m_{2}^{-1})}{m_{2}^{-1}-1}  \right)=:U_2.
\end{align*}
Here, the equalities come from Theorem \ref{cai_lim}, and the inequality comes from equation \eqref{opt_u_b_value}. We also have that 
\begin{align*}
  &\hat{D}_{KL}(P\| Q)=\inf_{\alpha\in\Phi^+_{dens}(P,n)} D_{KL}(\alpha\| Q)\\ 
  &\leq \inf_{\alpha\in\Phi^+_{dens}(P,n)} {d}_{TV}(\alpha, Q)\left( \frac{\log(M_{1}^{-1})}{1-M_{1}^{-1}} + \frac{\log(m_{1}^{-1})}{m_{1}^{-1}-1}  \right)\\
  &=\hat{d}_{TV}(P, Q)\left( \frac{\log(M_{1}^{-1})}{1-M_{1}^{-1}} + \frac{\log(m_{1}^{-1})}{m_{1}^{-1}-1}  \right)=:U_1.
\end{align*}
Once more, the equalities come from Theorem \ref{cai_lim}, and the inequality comes from equation \eqref{opt_u_b_value}. Hence, by selecting the largest between $U_1$ and $U_2$ we find the desired (optimal) upper bound for $\hat{D}_{KL}(P\| Q)$
\begin{align*}
    \hat{U}(\mathcal{A}(\delta,&m_1,m_2,M_1,M_2))=\\
    \max\bigg\{ &\hat{d}_{TV}(P,Q)\left( \frac{\log(M_{1}^{-1})}{1-M_{1}^{-1}} + \frac{\log(m_{1}^{-1})}{m_{1}^{-1}-1}  \right) ,\\ 
    &\hat{d}_{TV}(P,Q)\left( \frac{\log(M_{2}^{-1})}{1-M_{2}^{-1}} + \frac{\log(m_{2}^{-1})}{m_{2}^{-1}-1}  \right)  \bigg\}.
\end{align*}
Notice that in the case where $m_1=1$ or $M_1=1$, then $U_1$ is understood as being equal to $0$. A similar reasoning holds for the case where $m_2=1$ or $M_2=1$, with $U_2$ in place of $U_1$.
\end{proof}

\item[(IV)] Finally, we show that $\text{pol}_{\hat{d}_{TV}} \leq \hat{L}(\delta)$. We have that
\begin{align*}
    &\hat{L}(\delta):=\inf_{\hat{d}_{TV}(P,Q)=\delta}\hat{D}_{KL}(P\| Q)\\
    &=\inf_{\hat{d}_{TV}(P,Q)=\delta}\inf_{\beta\in\Phi^-(Q,d)}{D}_{KL}(P\| \beta)\\
    &\geq \inf_{\hat{d}_{TV}(P,Q)=\delta}\inf_{\beta\in\Phi^-(Q,d)}\bigg[\frac{1}{2}d_{TV}(P,\beta)^2 \\
    &+\frac{1}{36}d_{TV}(P,\beta)^4 +\frac{1}{270}d_{TV}(P,\beta)^6\\
    &+\frac{221}{340200}d_{TV}(P,\beta)^8 \bigg]\\
    &\geq \inf_{\hat{d}_{TV}(P,Q)=\delta} \bigg[\frac{1}{2} \hat{d}_{TV}(P,Q)^2+\frac{1}{36}\hat{d}_{TV}(P,Q)^4\\ &+\frac{1}{270}\hat{d}_{TV}(P,Q)^6+\frac{221}{340200}\hat{d}_{TV}(P,Q)^8\bigg]\\
    &=\frac{1}{2}\delta^2+\frac{1}{36}\delta^4+\frac{1}{270}\delta^6+\frac{221}{340200}\delta^8.
\end{align*}
Here, the first equality comes from definition \eqref{L_hat}, the second equality comes from Theorem \ref{cai_lim}, the first inequality comes from Theorem \ref{thm2}, and the second inequality comes from the fact that the infimum of a sum is not smaller than the sum of the infima. The last equality comes from our assumption that $\hat{d}_{TV}(P,Q)=\delta$. Notice that if we substitute $\inf_{\beta\in\Phi^-(Q,d)}$ with $\inf_{\alpha\in\Phi^+_{dens}(P,n)}$ the proof still holds thanks to Theorem \ref{cai_lim}.
%where the inequality is a consequence of the properties of the infimum operator and of Theorem \ref{thm2}. The statement follows.
\end{enumerate}
\end{proof}

Theorem \ref{main3} is extremely important because for a given value $\delta$ of the augmented TV distance between two generic distributions, it gives us immediately a lower bound for the augmented KL divergence. In addition, if the essential suprema and essential infima in \eqref{ess1} and \eqref{ess2} are well defined, Theorem \ref{main3} also gives an upper bound for the augmented KL divergence. The next example gives another interesting byproduct of our main result.

%it gives an interval for the value of the augmented KL divergence between two generic (essentially bounded) distributions $P\in M(\mathbb{R}^d)$ and $Q\in M(\mathbb{R}^n)$, $n\geq d$, whose augmented TV distance is known to be equal to some $\delta\in [0,1]$. 
\begin{example}\label{ex1}
Consider a one-dimensional Gaussian distribution and write $\rho_1=\mathcal{N}(\mu,\sigma^2)$, where $\mu\in\mathbb{R}$ is the mean and $\sigma^2>0$ is the variance. Consider then an $n$-dimensional Gaussian distribution and write $\rho_2=\mathcal{N}_n(\nu,\Sigma)$, where $\nu\in\mathbb{R}^n$ is the mean vector and $\Sigma \in \mathbb{R}^{n\times n}$ is the covariance matrix. Call $\zeta_1$ and $\zeta_n$ the largest and smallest eigenvalues of $\Sigma$, respectively. Then, \cite[Example VI.2]{lim2} shows that
$$\hat{D}_{KL}(\rho_1\|\rho_2)=\begin{cases}
    \frac{1}{2} \left[ \frac{\sigma^2}{\zeta_n} -1 + \log \left( \frac{\zeta_n}{\sigma^2} \right) \right] & \text{if } \sigma < \sqrt{\zeta_n}\\
    \frac{1}{2} \left[ \frac{\sigma^2}{\zeta_1} -1 + \log \left( \frac{\zeta_1}{\sigma^2} \right) \right] & \text{if } \sigma > \sqrt{\zeta_n}\\
    0 & \text{otherwise}
\end{cases}.$$
Call now $\xi$ the value taken by $\hat{D}_{KL}(\rho_1\|\rho_2)$. Then, by Theorem \ref{main3}, we find an upper bound to $\hat{d}_{TV}(\rho_1,\rho_2)$ by solving
$$\frac{1}{2}\delta^2+\frac{1}{36}\delta^4+\frac{1}{270}\delta^6+\frac{221}{340200}\delta^8 \leq \xi,$$
where $\delta=\hat{d}_{TV}(\rho_1,\rho_2)$. 

If instead we let $\rho_1$ and $\rho_2$ be a truncated one- and $n$-dimensional Gaussian, respectively, then we can use  $\hat{D}_{KL}(\rho_1\|\rho_2) \leq \hat{U}(\mathcal{A}(\delta,m_1,m_2,M_1,M_2))$ from Theorem \ref{main3} and
\begin{align*}
    \hat{U}(\mathcal{A}(\delta,&m_1,m_2,M_1,M_2))=\\
    \max\bigg\{ &\hat{d}_{TV}(P,Q)\left( \frac{\log(M_{1}^{-1})}{1-M_{1}^{-1}} + \frac{\log(m_{1}^{-1})}{m_{1}^{-1}-1}  \right) ,\\ 
    &\hat{d}_{TV}(P,Q)\left( \frac{\log(M_{2}^{-1})}{1-M_{2}^{-1}} + \frac{\log(m_{2}^{-1})}{m_{2}^{-1}-1}  \right)  \bigg\}.
\end{align*}
from Claim \ref{claim_ub_aug} to find a lower bound for $\delta=\hat{d}_{TV}(\rho_1,\rho_2)$. Notice that in this case we need the Gaussians to be truncated to ensure the essential suprema and essential infima in \eqref{ess1} and \eqref{ess2} are well defined.
%let $\rho_1=\mathcal{N}(s,S,\mu,\sigma^2)$ be a one dimensional Gaussian truncated between reals $s$ and $S$, and $\rho_2=\mathcal{N}_n(l,L,\nu,\Sigma)$ an $n$-dimensional Gaussian truncated between vectors $l$ and $L$, then we can compute the essential suprema $M_1$ and $M_2$ and essential infima $m_1$ and $m_2$ in \eqref{ess1} and \eqref{ess2}. In turn, this allows to compute 
\end{example}

Before concluding we point out that generalizing the proof that leads to equation \eqref{reid2} to the $\boldsymbol{\beta}\neq\boldsymbol{\beta}^\prime$ case is not easy; although we conjecture that a similar result holds, this will be the subject of future studies.

\section{Conclusion}\label{concl}
In this note, we presented optimal upper and lower bounds for the augmented KL divergence in terms of the augmented TV distance. This is just the first step towards a deep study of augmented divergences that ideally should include structural properties, statistical analysis, duality, possible applications, and many more aspects. We plan to be at the forefront of this process.

More concretely, in the near future we plan to find bounds for more augmented divergences in terms of augmented metrics and vice versa, in the spirit of \cite{gibbs}. It would be especially interesting to generalize \cite[Theorem 6]{reid} to the augmented framework of \cite{lim2}. An encouraging result of this kind  is presented in \cite[Corollary III.6]{lim2}: the authors give a bound for the augmented TV metric in terms of the augmented Hellinger squared divergence. We also plan to extend the second part of Claim \ref{claim_str_lb} to the $\boldsymbol{\beta}\neq\boldsymbol{\beta}^\prime$ case.

%It is worth to notice that in \cite[Corollary 1]{reid} the authors give an explicit value for $L(\delta)$, in contrast with \eqref{param_gamma} where the value is implicit. In particular, they show that given two pm's $P,Q$ on the same Euclidean space such that $d_{TV}(P,Q)=\delta\geq 0$, then
%$$L(\delta) = \min_{\gamma\in[\delta-2,2-\delta]} \left[ \left( \frac{\delta+2-\gamma}{4} \right) \log\left( \frac{\gamma-2-\delta}{\gamma-2+\delta} \right) + \left( \frac{\gamma+2-\delta}{4} \right) \log\left( \frac{\gamma+2-\delta}{\gamma+2+\delta} \right) \right].$$
%We will generalize this value to the augmented counterpart in future work.

\section*{Acknowledgements}\label{ackn}
We would like to thank Edric Tam, Yuhang Cai, and Vittorio Orlandi for their help with technical details, and Insup Lee, Oleg Sokolsky, Souradeep Dutta, Radoslav Ivanov, Kuk Jang, and Vivian Lin for inspiring this project and helpful discussions. Our deepest gratitude goes also to Sayan Mukherjee for covering the article processing charges and to two anonymous referees for their generous suggestions regarding content and presentation.

\bibliographystyle{plain}
\bibliography{References}

\begin{IEEEbiography}[{\includegraphics[width=1in,height=1.25in,clip,keepaspectratio]{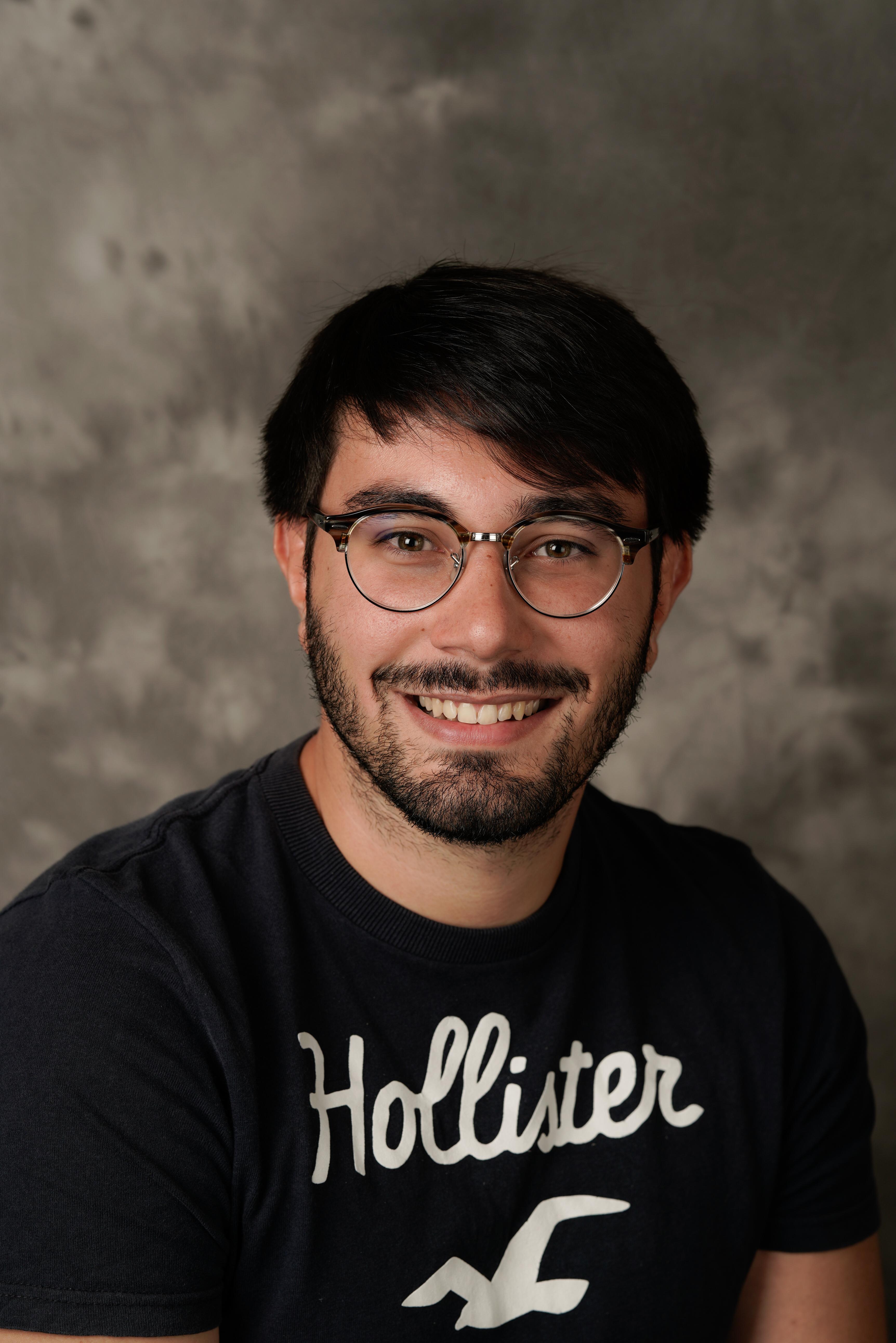}}]{Michele Caprio} received his BSc (in 2015) and MSc (in 2018) in Economics from Bocconi University in Milan, Italy, and his PhD (in 2022) in Statistics from Duke University in Durham, North Carolina, USA. 

He is a Postdoctoral Researcher at the PRECISE Center of the Department of Computer and Information Science of the University of Pennsylvania in Philadelphia, Pennsylvania, USA. His broad research interests are foundations of probability, mathematical statistics, and AI. More specifically, he is interested in imprecise probabilities and their applications to statistics and AI. 

Dr. Caprio was awarded the Aleane Webb Dissertation Research Fellowship and the IMS Hannan Travel Award; in 2022, he was a finalist for the NESS Student Research Award.
\end{IEEEbiography}

\EOD
\end{document}